\documentclass{amsart}

\usepackage{amsmath}
\usepackage{amsfonts}
\usepackage{amssymb,enumerate}
\usepackage{amsthm}
\usepackage[all]{xy}
\usepackage{rotating}
\usepackage{hyperref}

\theoremstyle{plain}
\newtheorem{lem}{Lemma}[section]

\newtheorem{prop}[lem]{Proposition}
\newtheorem{thm}[lem]{Theorem}

\theoremstyle{definition}
\newtheorem{defn}[lem]{Definition}
\newtheorem{ex}[lem]{Example}
\newtheorem{question}[lem]{Question}

\newtheorem{disc}[lem]{Remark}

\newtheorem{convention}[lem]{Convention}

\newcommand{\id}{\operatorname{id}}

\newcommand{\ann}{\operatorname{Ann}}

\newcommand{\HH}{\operatorname{H}}
\newcommand{\Hom}{\operatorname{Hom}}

\newcommand{\shift}{\mathsf{\Sigma}}

\newcommand{\cone}{\operatorname{Cone}}

\newcommand{\ideal}[1]{\mathfrak{#1}}
\newcommand{\m}{\ideal{m}}
\newcommand{\n}{\ideal{n}}

\newcommand{\comp}[1]{\widehat{#1}}

\newcommand{\bbz}{\mathbb{Z}}
\newcommand{\bbn}{\mathbb{N}}

\newcommand{\from}{\leftarrow}
\newcommand{\xra}{\xrightarrow}

\newcommand{\onto}{\twoheadrightarrow}

\newcommand{\vf}{\varphi}

\newcommand{\x}{\mathbf{x}}

\renewcommand{\geq}{\geqslant}
\renewcommand{\leq}{\leqslant}

\newcommand{\Ext}[4][R]{\operatorname{Ext}_{#1}^{#2}(#3,#4)}

\newcommand{\Otimes}[3][R]{#2\otimes_{#1}#3}
\renewcommand{\Hom}[3][R]{\operatorname{Hom}_{#1}(#2,#3)}

\newcommand{\K}{K}
\newcommand{\KK}{\widetilde K}

\numberwithin{equation}{lem}

\begin{document}

\bibliographystyle{amsplain}

\author{Benjamin J. Anderson}

\author{Sean Sather-Wagstaff}

\address{Department of Mathematics,
NDSU Dept \# 2750,
PO Box 6050,
Fargo, ND 58108-6050
USA}

\email{benjamin.j.anderson@ndsu.edu}

\email{sean.sather-wagstaff@ndsu.edu}

\urladdr{http://www.ndsu.edu/pubweb/\~{}benjaand/}

\urladdr{http://www.ndsu.edu/pubweb/\~{}ssatherw/}

\thanks{This material is based on work supported by North Dakota EPSCoR and National 
Science Foundation Grant EPS-0814442.}

\title{NAK for Ext and Ascent of module structures}

\date{\today}

\dedicatory{To Roger A. Wiegand on the occasion of his retirement}



\begin{abstract}
We investigate the interplay between properties of Ext modules
and ascent of module structures along local ring homomorphisms.
Specifically, let $\vf\colon (R,\m,k)\to (S,\m S,k)$ be a flat local ring homomorphism.
We show that if $M$ is a finitely generated $R$-module such that
$\Ext{i}{S}{M}$ satisfies NAK (e.g. if $\Ext{i}{S}{M}$ is finitely generated over $S$) for $i=1,\ldots,\dim_{R}(M)$, then $\Ext{i}{S}{M}=0$ for all $i\neq0$ and $M$ has an $S$-module structure that is 
compatible with its $R$-module structure via $\varphi$.
We  provide  explicit  computations of $\Ext{1}{S}{M}$ to indicate how large it
can be when $M$ does not have a compatible $S$-module~structure.
\end{abstract}

\maketitle

\section{Introduction} \label{sec0}
Throughout this paper $(R, \mathfrak m, k)$ and $(S,\n,l)$ are commutative noetherian 
local rings. Given an $R$-module $M$, the $\m$-adic completion of $M$
is denoted $\comp M$. 

The genesis for this paper begins with the following result of Buchweitz
and Flenner~\cite[Theorem 2.3]{buchweitz:psrp}.

\begin{thm}  \label{thm110620a}
Let $M$ be an $R$-module. If $M$ is $\m$-adically complete, then 
for each flat $R$-module $F$ one has $\Ext{i}{F}{M}=0$ for all $i\geq1$.
\end{thm}

The converse fails with no extra assumptions on $M$, as we see next. 

\begin{ex}
Let $M$ be a non-zero injective $R$-module, and 
assume that $R$ has positive depth (e.g.,  $R$ is a domain and not a field).
Then $\Ext{i}{-}{M}=0$ for all $i\geq1$.
However, the
fact that $M$ is injective implies that it is divisible. Thus, we have
$M=\mathfrak{m}M$, and it follows that $\comp{M}=0\neq M$, so $M$ is not complete.
\end{ex}

The next result 
of Frankild and Sather-Wagstaff~\cite[Corollary~3.5]{frankild:dcev} shows that the converse 
to Theorem~\ref{thm110620a} does hold when $M$ is finitely generated.

\begin{thm}  \label{cor110620}
Let $M$ be a finitely generated $R$-module. Then the following conditions are equivalent.
\begin{enumerate}[\rm(i)]
\item \label{cor110620a}
$M$ is $\mathfrak{m}$-adically complete.
\item $\Ext{i}{F}{M}=0$ for all $i\geq1$ for each flat $R$-module $F$.
\item $\Ext{i}{\comp{R}}{M}=0$ for all $i\geq1$.
\end{enumerate}
\end{thm}

It is worth noting that the proof of this result is quite technical, relying heavily on the  machinery of
derived local homology and derived local cohomology.

Since the module $M$ in Theorem~\ref{cor110620} is finitely generated, a standard result shows that
condition~\eqref{cor110620a} is equivalent to the following:
$M$ has an $\comp{R}$-module structure that is compatible with its $R$-module structure 
via the natural map $R\to\comp{R}$.
Thus, we consider the following  ascent question, focusing on homological conditions.

\begin{question}\label{qst110702}
Given a ring homomorphism $\varphi\colon R\to S$ what conditions on an $R$-module $M$ guarantees that it has an $S$-module structure compatible with its $R$-module structure via $\varphi$?
\end{question}

In the setting of Question~\ref{qst110702}, 
the module $\Otimes SM$ has a natural $S$-module
structure. Thus, if one had an $R$-module isomorphism 
$M\cong\Otimes SM$, then one could transfer
the  $S$-module structure from $\Otimes SM$ to $M$.
One can similarly inflict an $S$-module structure on $M$ if $M\cong\Hom SM$.
The following result of Frankild, Sather-Wagstaff, and R.~Wiegand~\cite[Main Theorem 2.5]{frankild:amsveem} 
and Christensen and Sather-Wagstaff~\cite[Theorem 3.1 and Remark 3.2]{christensen:tgdrh} 
shows that, when $\vf$ is a local homomorphism
with properties like those of the natural map $R\to\comp R$,
these are in fact the only way for $M$ to admit a 
compatible $S$-module structure.

\begin{thm}  \label{110620c}
Let $\varphi\colon R\to S$ be a flat local ring homomorphism such that 
the induced map $R/\mathfrak{m}\to S/\mathfrak{m}S$ is an isomorphism.
Let $M$ be a finitely generated $R$-module, and consider  the following conditions:
\begin{enumerate}[\rm(i)]
\item \label{110620c1}
$M$ has an $S$-module structure compatible with its $R$-module structure via $\varphi$.
\item \label{110620c2}
$\Ext{i}{S}{M}=0$ for all $i\geq1$.
\item \label{110620c3}
$\Ext{i}{S}{M}$ is finitely generated over $R$ for  $i=1,\ldots,\dim(M)$.
\item \label{110620c4}
the natural map $\Hom{S}{M}\to M$ given by $f\mapsto f(1)$ is an isomorphism.
\item \label{110620c5}
the natural map $M\to\Otimes{S}{M}$ given by $m\mapsto 1\otimes m$ is an isomorphism.
\item \label{110620c6}
$\Otimes{S}{M}$ is finitely generated over $R$.
\item \label{110620c7}
$\Ext{i}{S}{M}$ is finitely generated over $S$ for  $i=1,\ldots,\dim_R(M)$.
\end{enumerate}
Then conditions~\eqref{110620c1}--\eqref{110620c6} are equivalent and imply
condition~\eqref{110620c7}. 
If $R$ is Gorenstein, then conditions~\eqref{110620c1}--\eqref{110620c7} are equivalent.
\end{thm}

The proof of this result is less technical than that of Theorem~\ref{cor110620}.
But it does use the Amplitude Inequality of Foxby, Iyengar, and 
Iversen~\cite{foxby:daafuc,iversen:aifc}---a consequence of the New Intersection Theorem---and derived Gorenstein injective dimension.

The next result is the main theorem of the current paper. It contains several
improvements on Theorem~\ref{110620c}.
First, it removes the Gorenstein hypothesis for the implication
$\eqref{110620c7}\implies\eqref{110620c1}$.
Second, it further relaxes the conditions on the Ext-modules needed to obtain
an $S$-module structure on $M$.
Third, the proof is significantly less technical than the proofs of these earlier results,
relying only on basic properties of  Koszul complexes.
It is proved in Theorem~\ref{thm110619a}.

\begin{defn}\label{defn 110211}
An $R$-module $N$ \emph{satisfies NAK} if  $N=0$ or $N/\mathfrak{m}N\neq0$. 
\end{defn}

\begin{thm}\label{thm110703a} 
In Theorem~\ref{thm110619a}
the conditions~\eqref{110620c1}--\eqref{110620c7} are equivalent, and they are equivalent to the
following:
\begin{enumerate}[\rm(viii)]
\item \label{110620c8}
$\Ext{i}{S}{M}$ satisfies NAK for  $i=1,\ldots,\dim_R(M)$.
\end{enumerate}
\end{thm}

We conclude this introduction by outlining the contents of the paper.
Section~\ref{sec110619a} summarizes foundational material needed for the
proof of Theorem~\ref{thm110619a}.
Section~\ref{sec110619b} is devoted to the proof of this result.
Finally, Section~\ref{sec111121a} consists of an example demonstrating
how large $\Ext iSM$ is, even over a relatively small ring.

\section{Backround Material}\label{sec110619a}

Most of our definitions and notational conventions come 
from~\cite{christensen:gd,foxby:hacr}. For the sake of clarity, we specify
a few items here.

\begin{convention}\label{notn111114a}
We index chain complexes of $R$-modules (``$R$-complexes'' for short)
homologically:
$$Y=\cdots\xra{\partial^Y_{i+1}}Y_i\xra{\partial^Y_{i}}Y_{i-1}\xra{\partial^Y_{i-1}}\cdots.$$
Given an $R$-complex $Y$ and an integer $n$, the \emph{$n$th suspension}
(or \emph{shift}) of $Y$
is denoted $\shift^nY$ and has $(\shift^nY)_i:=Y_{i-n}$.
We set $\shift Y:=\shift^1Y$.

For two $R$-complexes $Y$ and $Z$, our  indexing protocol
extends to the Hom-complex $\Hom YZ$, so that we have
$\Hom YZ_i=\prod_{j\in\bbz}\Hom{Y_j}{Z_{j+i}}$ for each $i\in\bbz$.
As usual, 
the tensor product $\Otimes YZ$ has
$(\Otimes YZ)_i=\prod_{j\in\bbz}\Otimes{Y_j}{Z_{i-j}}$.

Given an element $x\in R$, the \emph{Koszul complex} $K^R(x)$ is
$0\to R\xra xR\to 0$, concentrated in degrees 0 and 1.
If $\x=x_1,\ldots,x_n$ is a sequence in $R$, then the \emph{Koszul complex}
$K^R(\x)$ is defined inductively: $K^R(\x)=\Otimes{K^R(x_1,\ldots,x_{n-1})}{K^R(x_n)}$.

For each morphism of $R$-complexes (i.e., a chain map)
$f\colon Y\to Z$, the mapping cone $\cone(f)$ gives rise to a 
short exact sequence
$0\to Z\to\cone(f)\to \shift Y\to 0$, hence an associated long exact sequence
on homology. Given an element $x\in R$, 
the mapping cone associated to the morphism $Y\xra{x}Y$ is isomorphic to 
$\Otimes{K^R(x)}{Y}$.
A morphism $f\colon Y\to Z$ is a \emph{quasiisomorphism} if the induced map
$\HH_i(f)\colon\HH_i(Y)\to \HH_i(Z)$ on homology is an isomorphism for each $i\in\bbz$.
\end{convention}

\begin{lem}\label{lem110619a}
Let $i_0$ be a fixed integer, let $\textbf{x}=x_1,\ldots,x_n$ be a sequence in $R$, and let $Y$ be an $R$-complex such that $\HH_{i}(Y)=0$ for all $i<i_0$. Then $\HH_i(\Otimes{K^{R}(\textbf{x})}{Y})=0$ for all $i<i_0$ and $\HH_{i_0}(\Otimes{K^{R}(\textbf{x})}{Y})\cong\HH_{i_0}(Y)/(\textbf{x})\HH_{i_0}(Y)$.
\end{lem}

\begin{proof}
Argue by induction on $n$, using the long exact sequence 
in homology associated to the short
exact sequence $0\to Y\to \Otimes{K^R(x_1)}{Y}\to\shift Y\to 0$.
\end{proof}

\begin{lem}\label{lem110619e}
Let $\vf\colon R\to S$ be a flat local ring homomorphism such that the induced map
$R/\m\to S/\m S$ is an isomorphism.
Let $\x=x_1,\ldots,x_n\in R$ be a generating sequence for $\m$.
\begin{enumerate}[\rm(a)]
\item \label{lem110619e1}
The natural map $K^{R}(\textbf{x})\to\Otimes{S}{K^{R}(\textbf{x})}$ is a quasi-isomorphism. 
\item \label{lem110619e2}
For each bounded above complex $J$ of injective $R$-modules,
the induced map $\Hom{\Otimes{S}{K^{R}(\textbf{x})}}{J}\to\Hom{K^{R}(\textbf{x})}{J}$ is a quasi-isomorphism.
\end{enumerate}
\end{lem}

\begin{proof}
Part~\eqref{lem110619e1} is from~\cite[2.3]{frankild:amsveem},
and part~\eqref{lem110619e2} follows from this by a standard property of 
bounded above complexes of injective $R$-modules.
\end{proof}

The next  result is  from~\cite[proof of  Theorem 2.5]{frankild:amsveem}.

\begin{lem} \label{lem110619c}
Let $\vf\colon R\to S$ be a flat  ring homomorphism, and let $M$ be an $R$-module. 
Then $\Ext{i}{S}{M}=0$ for all $i>\dim(R/\ann_R(M))$.
In particular, if $M$ is finitely generated, then $\Ext{i}{S}{M}=0$ for all $i>\dim_{R}(M)$.
\end{lem}

\section{Proof of Theorem~\ref{thm110703a}} \label{sec110619b}

Let $\vf\colon R\to S$ be a local ring homomorphism.
The term ``satisfies NAK'' is defined in~\ref{defn 110211}.
Given an $S$-module $N$, 
if $N$ satisfies NAK as an $S$-module, then it satisfies NAK as an $R$-module
because of the epimorphism $N/\m N=N/\m SN\onto N/\n N$.
Furthermore, this reasoning shows that the converse holds if $\m S=\n$,
e.g., if the induced map $R/\m \to S/\m S$ is an isomorphism.

\begin{lem}\label{lem110907a}
Let $\vf\colon R\to S$ be a flat local ring homomorphism  such that the induced map
$R/\m\to S/\m S$ is an isomorphism. Let $M$ be an $R$-module, and let $z\geq 1$. If $\Ext{i}{S}{M}=0$ for all $i>z$ and $\Ext{z}{S}{M}$ satisfies NAK then $\Ext{z}{S}{M}=0$.
\end{lem}

\begin{proof}
Let $\x=x_1,\ldots,x_n\in R$ be a generating sequence for $\m$,
and let $J$ be an $R$-injective resolution of $M$. By assumption we have 
$\HH_{-i}(\Hom{S}{J})\cong\Ext{i}{S}{M}=0$ 
for all $i>z$, so  we have
$\HH_{-i}(\Otimes{K^R(\mathbf{x})}{\Hom{S}{J}})=0$
for all $i>z$ and 
\begin{align*}
\HH_{-z}(\Otimes{K^R(\mathbf{x})}{\Hom{S}{J}}) & \cong  \HH_{-z}(\Hom{S}{J})/\mathbf{x}\HH_{-z}(\Hom{S}{J})\\
& \cong  \Ext{z}{S}{M}/\mathbf{x}\Ext{z}{S}{M}
\end{align*}
by Lemma~\ref{lem110619a}.

We claim that $\Ext{z}{S}{M}/\mathbf{x}\Ext{z}{S}{M}=0$. To see this we first notice that 
\begin{align*}
\Otimes{K^R(\mathbf{x})}{\Hom{S}{J}} &\cong \Hom{\Hom{K^R(\mathbf{x})}{R}}{\Hom{S}{J}}\\
&\cong \Hom{\shift^{-n}K^R(\mathbf{x})}{\Hom{S}{J}}\\
&\cong\shift^{n}\Hom{K^R(\mathbf{x})}{\Hom{S}{J}}\\
&\cong\shift^{n}\Hom{\Otimes{S}{K^R(\mathbf{x})}}{J}\\
&\simeq\shift^{n}\Hom{K^R(\mathbf{x})}{J}\\
&\cong\Hom{\shift^{-n}K^R(\mathbf{x})}{J}\\
&\cong\Hom{\Hom{K^R(\mathbf{x})}{R}}{J}\\
&\cong\Otimes{K^R(\mathbf{x})}{J}.
\end{align*}
The first and last isomorphisms are Hom-evaluation; see, e.g., 
\cite[B.2. Lemma]{avramov:edcrcvct}. 
The fourth isomorphism is Hom-tensor 
evaluation, and the other isomorphisms follow from properties of the Koszul complex.
The quasiisomorphism in the fifth step is from Lemma~\ref{lem110619e}\eqref{lem110619e2}. 

This sequence explains the second isomorphism in the next display:
\begin{align*}
\Ext{z}{S}{M}/\mathbf{x}\Ext{z}{S}{M}
&\cong \HH_{-z}(\Otimes{K^R(\mathbf{x})}{\Hom{S}{J}})\\
&\cong\HH_{-z}(\Otimes{K^R(\mathbf{x})}{J})
\\
&\cong\HH_{-z}(\Otimes{K^R(\mathbf{x})}{M})\\
&=0.
\end{align*}
The first isomorphism comes from the first paragraph of this proof.
The second isomorphism follows from the fact that
the quasiisomorphism $M\xra\simeq J$ is respected by 
$\Otimes{K^R(\mathbf{x})}-$.
And the vanishing is due to the assumption $z\geq 1$.

As $\Ext{z}{S}{M}$ satisfies NAK, the claim implies $\Ext{z}{S}{M}=0$ as desired. 
\end{proof}

Here is Theorem~\ref{thm110703a} from the introduction.

\begin{thm}\label{thm110619a}
Let $\vf\colon R\to S$ be a flat local ring homomorphism such that the induced map $R/\mathfrak{m}\to S/\mathfrak{m}S$ is an isomorphism, and let $M$ be a finitely generated $R$-module. If $\Ext{i}{S}{M}$ satisfies NAK for $i=1,\ldots,\dim_{R}(M)$ then $\Ext{i}{S}{M}=0$ for all $i\neq0$ and $M$ has an $S$-module structure compatible with its $R$-module structure via $\vf$.
\end{thm}

\begin{proof}
Set $z=\sup\{i\geq 0\mid\Ext iSM\neq 0\}$.\footnote{We take the supremum here instead of the
maximum since we do not know \emph{a priori} whether the set $\{i\geq 0\mid\Ext iSM\neq 0\}$
is non-empty.}
Lemma~\ref{lem110619c} implies that $z\leq\dim_R(M)$. 
Note that $z\leq 0$: if $z\geq 1$, then Lemma~\ref{lem110907a} implies that
$\Ext zSM=0$, a contradiction.
It follows that $\Ext{i}{S}{M}=0$ for all $i\neq0$, and the remaining conclusions follow from
Theorem~\ref{110620c}.
\end{proof}

\begin{disc}\label{disc111116b}
As we note in the introduction, our proof of this theorem removes the need to invoke
the Amplitude Inequality in the proof of Theorem~\ref{110620c}. Indeed, the 
Amplitude Inequality is used in the implication~\eqref{110620c3}$\implies$\eqref{110620c2},
which we prove directly in the proof of Theorem~\ref{thm110619a}.
\end{disc}

\begin{disc}\label{disc111123a}
One can paraphrase Theorem~\ref{thm110619a} as follows:
In Theorem~\ref{110620c}\eqref{110620c3} one can replace 
the phrase ``is finitely generated over $R$'' 
with the phrase ``satisfies NAK''. It is natural to ask whether the same replacement
can be done in Theorem~\ref{110620c}\eqref{110620c6}.
In fact, this cannot be done because, given a finitely generated $R$-module
$M$, the $S$-module $\Otimes SM$ is finitely generated, so it automatically
satisfies NAK, regardless of whether $M$ has a compatible $S$-module structure.
\end{disc}

\section{Explicit Computations} \label{sec111121a}

Given a ring homomomorphism $\vf\colon R\to S$ as in Theorem~\ref{thm110619a},
if $M$ is a finitely generated $R$-module that does \emph{not} have a compatible
$S$-module structure, then we know that $\Ext iSM$ does not satisfy NAK for some $i$.
Hence, this Ext-module is quite large. 
This section is devoted to a computation showing
how large this Ext-module is when 
$R\neq \comp R=S$, even for the simplest ring $R$, e.g., for
$R=k[X]_{(X)}$ where $k$ is a field or for $\bbz_{p\bbz}$.
See Remark~\ref{disc111130a}.

\begin{lem}\label{lem110601a}
Let $\vf\colon R\to S$ be a faithfully flat ring homomorphism, and let $C$ be an $R$-module.
Let $\m$ be a maximal ideal of $R$, and assume that $C$ is $\m$-adically complete.
\begin{enumerate}[\quad\rm(a)]
\item \label{lem110601a1} 
Then $\Ext i{S/R}{C}=0=\Ext iSC$ for all $i\geq 1$.
\item \label{lem110601a2} 
If $R$ is local and the natural map $R/\m\to S/\m S$ is an isomorphism, then $\Hom{S/R}{C}=0$,
and $C$ has an $S$-module structure compatible with its $R$-module structure via $\vf$,
and the natural  maps $C\to \Hom SC\to C$ are inverse isomorphisms.
\item \label{lem110601a3} 
If $R$ is local, then $\Ext{i}{\comp R/R}{\comp R}=0=\Ext{i+1}{\comp R}{\comp R}$ for all $i\geq 0$, and the natural  maps
$\comp R\to\Hom{\comp R}{\comp R}\to\comp R$ are inverse isomorphisms.
\end{enumerate}
\end{lem}

\begin{proof}
\eqref{lem110601a1}
The fact that $S$ is faithfully flat over $R$ implies that $\vf$ is a pure monomorphism,
and it follows that $S/R$ is flat over $R$; see~\cite[Theorem 7.5]{matsumura:crt}.
Since $S$ and $S/R$ are flat over $R$, the desired vanishing 
follows from Theorem~\ref{thm110620a}.

\eqref{lem110601a2}
Assume that $R$ is local and the natural map $R/\m\to S/\m S$ is an isomorphism.
In particular, the ideal $\m S\subset S$ is maximal.

Claim: the natural map $\comp\vf\colon\comp R\to\comp S$ between 
$\m$-adic completions is an isomorphism.
To see this, first note that the fact that $\m S$ is maximal implies that
$\comp S$ is local with maximal ideal $\m \comp S$. 
Furthermore, the induced map $\comp R/\m\comp R\to \comp S/\m\comp S$
is equivalent to the isomorphism $R/\m\to S/\m S$, so it is an isomorphism.
It follows from a version of Nakayama's Lemma~\cite[Theorem 8.4]{matsumura:crt}
that $\comp S$ is a cyclic $\comp R$-module. Since it is also faithfully flat,
we deduce that $\comp\vf$
is an isomorphism, as claimed.

For each $n\in\bbn$, the induced map $R/\m^n\to S/\m^nS$ is an 
isomorphism: indeed, this map is equivalent to the 
induced map $\comp R/\m^n\to \comp S/\m^n\comp S$ which is an isomorphism
because $\comp R\xra\cong\comp S$.
This justifies the last step in the following display:
\begin{align*}
\Otimes{(S/R)}{(R/\m^n)}
&\cong(\Otimes{S}{(R/\m^n)})/(\Otimes{R}{(R/\m^n)})
\cong(S/\m^nS)/(R/\m^n)=0.
\end{align*}
This display explains the fifth isomorphism in the next sequence:
\begin{align*}
\Hom{S/R}{C}
&\cong\Hom{S/R}{\lim_{n\from}C/\m^nC}\\
&\cong\lim_{n\from}\Hom{S/R}{C/\m^nC}\\
&\cong\lim_{n\from}\Hom{S/R}{\Hom[R/\m^n]{R/\m^n}{C/\m^nC}}\\
&\cong\lim_{n\from}\Hom[R/\m^n]{\Otimes{(R/\m^n)}{(S/R)}}{C/\m^nC}\\
&\cong\lim_{n\from}\Hom[R/\m^n]{0}{C/\m^nC}\\
&=0.
\end{align*}
Now consider the exact sequence $0\to R\to S\to S/R\to 0$ and part of the long exact sequence in $\Ext{}{-}{C}$.
$$
0\to\underbrace{\Hom{S/R}{C}}_{=0}\to\Hom{S}{C}\to\underbrace{\Hom{R}{C}}_{\cong C}\to\underbrace{\Ext{1}{S/R}{C}}_{=0}
$$
It follows that the induced map $\alpha\colon\Hom SC\to C$ is an isomorphism. 
It is straightforward to show that this is the evaluation map $f\mapsto f(1)$.
Since $C$ is complete, the isomorphism $\comp R\cong \comp S$ implies that $C$ has an $S$-module structure that is compatible with its $R$-module structure via $\vf$.
From this, it follows that the map $\beta\colon C\to\Hom SC$ given by $c\mapsto(s\mapsto sc)$ is a well-defined $S$-module homomorphism.
Since the composition $\alpha\beta$ is the identity on $C$, it follows that $\alpha$ and $\beta$ are inverse isomorphisms.

\eqref{lem110601a3}
This follows from parts~\eqref{lem110601a1} and~\eqref{lem110601a2} using $C=\comp R$.
\end{proof}

\begin{prop}\label{prop110526a}
Assume that  $R$ is a discrete valuation ring that is not complete, with $\m=XR$.
Set  $E=E_R(k)=E_{\comp R}(k)$,
and consider the quotient fields $\K=Q(R)$ and $\KK=Q(\comp R)$.
If $[\KK:\K]=\infty$, 
then there are an uncountable cardinal $C$ and $\comp R$-module isomorphisms
$$\Ext{i}{\comp R}{R}\cong
\begin{cases}
0 & \text{if $i\neq 1$} \\
E\oplus \KK^{(C)} & \text{if $i= 1$}
\end{cases}$$
where $\KK^{(C)}$ is the direct sum of copies of $\KK$ indexed by $C$.
\end{prop}

\begin{proof}
As a $K$-vector space and as an $R$-module, we have $\KK\cong K^{(A)}$ for some
infinite  cardinal $A$.
Note that since $R$ and $\comp R$ are discrete valuation rings with uniformizing parameter $X$, we have
$K\cong R_X$ and $\KK\cong \comp R_X\cong\Otimes{\K}{\comp R}$.
Since $K$ has no $X$-torsion, we have $\Hom{k}{K}=0$.

Claim 1: We have $\Ext i{\comp R}{R}=0$ for all $i\neq 1$.
Since $\id_R(R)=1$, it suffices to show that $\Hom{\comp R}{R}=0$.
From~\cite[Corollary 1.7]{frankild:amsveem} we know that
$\Hom{\comp R}{R}$ is isomorphic to a complete submodule  $I\subseteq R$.
Since $R$ is a discrete valuation ring, its non-zero submodules are all isomorphic to $R$,
which is not complete. So we must have $\Hom{\comp R}{R}\cong I=0$.

Claim 2: There is an $R$-module isomorphism $\comp R/ R\cong\KK/K$.
In the following commutative diagram, 
the top row is a minimal $R$-injective resolution of $R$,
and the bottom row is a minimal $\comp R$-injective resolution of $\comp R$:
\begin{equation}\label{eq110526a}
\begin{split}\xymatrix{
0\ar[r]
&R\ar[r]\ar[d]
&\K\ar[r]\ar[d]
&E \ar[r]\ar[d]
&0 \\
0\ar[r]
&\comp R\ar[r]
&\KK\ar[r]
& E\ar[r]
&0.}\end{split}\end{equation}
The Snake Lemma yields an $R$-module isomorphism $\comp R/ R\cong\KK/K$.

Claim 3: We have an $R$-module isomorphism $\Hom{\comp R}{\comp R/R}\cong (\KK/K)^{(B)}$
for some infinite cardinal $B>A$.
This is from the next sequence of $R$-module isomorphisms where the first step is from Claim 2:
\begin{align*}
\Hom{\comp R}{\comp R/R}
&\cong\Hom{\comp R}{\KK/K} \\
&\cong\Hom{\comp R}{\Hom[K]{K}{\KK/K}}\\
&\cong\Hom[K]{\Otimes{K}{\comp R}}{\KK/K} \\
&\cong\Hom[K]{\KK}{\KK/K} \\
&\cong\Hom[K]{K^{(A)}}{\KK/K} \\
&\cong\Hom[K]{K}{\KK/K}^A \\
&\cong(\KK/K)^A \\
&\cong (\KK/K)^{(B)}.
\end{align*}
The third step is by Hom-tensor adjointness,
and the fourth step is from the isomorphism $\KK\cong \Otimes{K}{\comp R}$ noted at the beginning of the proof, and the remaining steps are standard.
Since $A$ is infinite, we must have $B>A$, as claimed.

Claim 4: There is a cardinal $C$ and an $\comp R$-module isomorphism
$\Ext{1}{\comp R}{R}\cong
E\oplus \KK^{(C)}$.
We compute $\Ext{1}{\comp R}{R}$ using the injective resolution of $R$ from the top row of~\eqref{eq110526a}.
From Claim 1, this yields an exact sequence of $\comp R$-module homomorphisms
\begin{equation}\label{eq110526b}
0\to \Hom{\comp R}{K}\to\Hom{\comp R}{E}\to\Ext{1}{\comp R}{R}\to 0.
\end{equation}
Since  $K$ and $E$ are injective over $R$, the modules
$\Hom{\comp R}{K}$ and $\Hom{\comp R}{E}$ are injective over $\comp R$.
Because 
$\Hom{\comp R}{K}$  is injective over $\comp R$, the sequence~\eqref{eq110526b} splits.
As $\Hom{\comp R}{E}$ is injective over $\comp R$, it follows that $\Ext{1}{\comp R}{R}$ is injective over $\comp R$.
So, there is an $\comp R$-module isomorphism
\begin{equation}\label{eq110526c}
\Ext{1}{\comp R}{R}\cong E^{(D)}\oplus\KK^{(C)}
\end{equation}
for  cardinals $C$ and $D$ where $D=\dim_k(\Hom[\comp R]{k}{\Ext{1}{\comp R}{R}})$.
Since the sequence~\eqref{eq110526b} splits, we have the third step in the next sequence:
\begin{align*}
k
&\cong\Hom{k}E\\
&\cong\Hom[\comp R]{k}{\Hom{\comp R}{E}}\\
&\cong\Hom[\comp R]{k}{\Hom{\comp R}{K}}\oplus\Hom[\comp R]{k}{\Ext{1}{\comp R}{R}} \\
&\cong\Hom{k}{K}\oplus\Hom[\comp R]{k}{\Ext{1}{\comp R}{R}} \\
&\cong\Hom[\comp R]{k}{\Ext{1}{\comp R}{R}} 
\end{align*}
The second and fourth steps follow by Hom-tensor adjointness,
and the fifth step follows from the vanishing $\Hom{k}{K}=0$
noted at the beginning of the proof.
It follows that $D=1$, so the claim follows from the isomorphism~\eqref{eq110526c}.

Now we complete the proof of the proposition. Because of Claim 4,  we need only show that 
$C$ is uncountable.
Consider the exact sequence
$$0\to R\to \comp R\to\comp R/R\to 0$$
and part of the associated  long exact sequence induced by $\Ext{}{\comp R}{-}$: 
$$\Hom{\comp R}{R}\to\Hom{\comp R}{\comp R}\to\Hom{\comp R}{\comp R/R}\to
\Ext 1{\comp R}{R}\to\Ext 1{\comp R}{\comp R}.$$
Over $R$, this sequence
has the following form
by Claims 1 and 3 and Lemma~\ref{lem110601a}\eqref{lem110601a3}:
$$0\to\comp R\to(\KK/K)^{(B)}\to
\Ext 1{\comp R}{R}\to0.$$
Apply the  functor $(-)_X$ to obtain the  exact sequence of $\K$-module homomorphisms
$$0\to \KK\to(\KK/K)^{(B)}\to
\Ext 1{\comp R}{R}_X\to0$$
which therefore splits.
Since $E_X=0$, it follows from Claim 4 that over $R$ we have
$$\KK^{(C)}
\cong \Ext 1{\comp R}{R}_X
\cong(\KK/K)^{(B)}/\KK
\cong(K^{(A)})^{(B)}/K^{(A)}
\cong K^{(B)}/K^{(A)}
\cong K^{(B)}.$$
The last two steps in this sequence follow from the fact that $A$ and $B$ are infinite cardinals such that $B>A$.

Suppose that $C$ were countable. It would then follow that $C\leq A$, so we have
$$K^{(B)}
\cong \KK^{(C)}
\cong (K^{(A)})^{(C)}
\cong K^{(A)}.$$
It follows that $B=A$, contradicting the fact that $B>A$. It follows that $C$ is uncountable, 
as desired.
\end{proof}

\begin{disc}\label{disc111130a}
Nagata~\cite[(E3.3)]{nagata:lr} shows that the assumption
$[\KK:\K]=\infty$ in Proposition~\ref{prop110526a} is not automatically satisfied.
On the other hand, the next
result shows that the condition $[\KK:\K]=\infty$ 
is satisfied by the  rings $\bbz_{p\bbz}$ and $R=k[X]_{(X)}$.
\end{disc}

\begin{prop}\label{prop111130a}
Assume that $R$ is a discrete valuation ring 
with $\m=XR$ and such that $|R|=|k|$.
For the quotient fields $\K=Q(R)$ and $\KK=Q(\comp R)$, we have
$[\KK:\K]=\infty$.
\end{prop}

\begin{proof}
We claim that $|\comp R|>|R|$.
To show this, let $\{a_t\}_{t\in k}\subseteq R$ be a set of representatives of the elements of $k$.
Then every element of $\comp R$ has a unique representation 
$\sum_{i=0}^\infty a_{t_i}X^i$. It follows that
$|\comp R|=|k|^{\aleph_0}>|k|=|R|$, as claimed.

Suppose now that $[\KK:\K]=A<\infty$.
The fact that $K$ is infinite implies that
$$|K|=A|K|=|\KK|=|\comp R|>|R|=|K|$$
a contradiction.
\end{proof}

The proof of Proposition~\ref{prop110526a} translates directly to give the following.

\begin{prop}\label{prop110526a'}
Assume that  $R$ is a discrete valuation ring that is not complete, with $\m=XR$.
Set  $E=E_R(k)=E_{\comp R}(k)$,
and consider the quotient fields $\K=Q(R)$ and $\KK=Q(\comp R)$.
If $[\KK:\K]=A<\infty$, 
then there are $\comp R$-module isomorphisms
$$\Ext{i}{\comp R}{R}\cong
\begin{cases}
0 & \text{if $i\neq 1$} \\
E\oplus \KK^{A-2} & \text{if $i= 1$.}
\end{cases}$$
\end{prop}

\begin{disc}\label{disc111130b}
It is worth noting that, in the notation of Proposition~\ref{prop110526a'},
we cannot have $A=1$. Indeed, if $A=1$, then we have $\KK=\K$,
and the proof of Proposition~\ref{prop110526a'} shows that
$\comp R/R\cong\KK/\K=0$, contradicting the assumption that 
$R$ is not complete.

On the other hand,
Nagata~\cite[(E3.3)]{nagata:lr} shows how to build a ring $R$
such that $A=2$, which then has 
$\Ext{i}{\comp R}{R}\cong E$ by Proposition~\ref{prop110526a'}.
\end{disc}

\section*{Acknowledgments}

We are grateful 
to \texttt{a-fortiori@mathoverflow.net}, Phil Griffith, Bruce Olberding, Irena Swanson
for helpful discussions about this material. We also thank the anonymous referee
for  valuable comments and suggestions.

\providecommand{\bysame}{\leavevmode\hbox to3em{\hrulefill}\thinspace}
\providecommand{\MR}{\relax\ifhmode\unskip\space\fi MR }
\providecommand{\MRhref}[2]{%
  \href{http://www.ams.org/mathscinet-getitem?mr=#1}{#2}
}
\providecommand{\href}[2]{#2}

\end{document}